\documentclass[12pt]{amsart}

\usepackage{amsmath,amssymb,amsthm}
\usepackage{bbm}
\usepackage[all]{xy}
\usepackage[dvips]{graphicx}
\usepackage{latexsym}

\newtheorem{theorem}{Theorem}[section]
\newtheorem{lemma}[theorem]{Lemma}

\theoremstyle{definition}

\newtheorem{definition}[theorem]{Definition}

\def\R{\mathbbm{R}}

\def\D{\mathbbm{D}}
\def\C{\mathbbm{C}}
\def\s{\mathbbm{S}}
\def\B{\mathbbm{B}}

\def\0{\underline{0}}

\def\a{\alpha}
\def\e{\varepsilon}

\begin{document}

\title{Milnor Fibrations and  the Thom Property for maps $f \bar g$}
\author{Anne Pichon and  Jos\'e Seade }
\thanks{
{ Keywords: Whitney stratifications, Thom $a_f$ property,
real singularities, Milnor fibrations. \\
Research partially supported by CNRS (France) and  CONACYT (Mexico) through the Laboratoire International Associ\'e Solomon Lefschetz (LAISLA) and the ECOS-ANUIES program. }
}
\maketitle
\vskip0,2cm\noindent

\begin{abstract}
 We prove that every map-germ ${f \bar g}: (\C^n,\0) {\to}(\C,0)$ with an isolated critical value at $0$ has the Thom $a_{f \bar g}$-property. This extends Hironaka's theorem for holomorphic mappings
 to  the case of map-germs $f \bar g$ and it implies that  every such map-germ has a Milnor-L\^e fibration
defined on a Milnor tube.  One thus has a locally trivial fibration $\phi: \mathbb S_\e \setminus K \to \mathbb S^1$ for every sufficiently small sphere around $\0$,
 where $K$ is the link of $f \bar g$ and in a neighbourhood of $K$ the projection map  $\phi$ is given by $f \bar g / \vert f \bar g\vert$.
\end{abstract}

\section*{Introduction}

 Soon after J. Milnor published his  book \cite{Mi2}, there were several interesting articles  about Milnor fibrations for real singularities published
 by various people, as for instance by E. Looijenga, P. T. Church and K. Lamotke, N. A'Campo, B. Perron,  L. Kauffman and W. Neumann, A. Jacquemard  and others. More recently, there has been a new wave of interest in the topic and a number of articles have been published by various authors (see for instance \cite {AC, BP, BPS, CSS2, RT, Massey, Oka2, Pichon, PS2, RS, RSV,  Se3}).

Unlike the fibration theorem for complex singularities, which holds for every map-germ $(\C^n,0) \to (\C,0)$, in the real case one needs to impose stringent conditions  to get a fibration on a ``Milnor tube'', or a fibration on a sphere, as in the holomorphic case.

In \cite{PS2} we observed that L\^e's arguments in \cite{Le1} for holomorphic mappings extend to every real analytic map germ $(\R^n,\0) \to (\R^p,0)$, $n > p$, with an isolated critical value, provided it has the Thom $a_f$-property and  $V:= f^{-1}(0)$ has dimension more than 0. Hence one has in that setting a Milnor-L\^e fibration:
$$f: N(\e,\delta)  \to \D_{\delta} \setminus \{0\}\;.$$
Here $N(\e,\delta)$ denotes a ``solid Milnor tube'': it is the intersection $f^{-1}(\D_\delta \setminus \{0\}) \cap \B_\e$, where  $\B_\e$ is a sufficiently small ball around $\0 \in \R^n$ and $\D_{\delta}$ is a ball in $\R^p$ of radius small enough with respect to $\e$. This was later completed in \cite{CSS2} (see also \cite{RT}), giving necessary and sufficient conditions for one such map-germ to define a Milnor fibration on every small sphere around the origin, with projection map $f/\vert f \vert $.

Then, an interesting problem is finding families of map germs $(\R^n,\0) \to (\R^p,0)$, $n > p$, having an isolated critical value and the Thom property. This is even better when the given families further have a rich geometry one can use in order to study the  topology of the corresponding Milnor fibrations (cf. \cite{BPS}).

In this article we prove:

\medskip
\noindent
{\bf Theorem.} {\it Let  $f, g$ be holomorphic map germs $(\C^n,\0) \to (\C,0)$ such that the map $f \bar g$ has an isolated critical value at $0 \in \C$. Then $f \bar g$ has the  Thom $a_{f \bar g}$-property.  }

\medskip
In fact our proof  is of a local nature and therefore extends, with  same proof, to the case
of holomorphic map-germs defined on a complex analytic variety $X$ with an isolated singularity. This result  generalizes to higher dimensions the corresponding theorem in \cite{PS2} for $n=2$, and it has the following
 corollaries:

\medskip
\noindent
{\bf Corollary 1.}
 {\it
 Let  $f,g$ be holomorphic map-germs defined on a complex analytic variety $X$ with an isolated singularity at a point $\0$, such that the germ $ f \bar g$ has an isolated critical value at $0$.
Then one has a locally trivial fibration $$\qquad \qquad \qquad \qquad \qquad \qquad \quad  N(\e,\delta) \buildrel {f} \over {\longrightarrow} \D_\delta \setminus \{0\} \,, \quad \e \gg \delta> 0 \; \hbox{sufficiently small} \,,
$$
where  $N(\e,\delta) := [(f \bar g)^{-1}(\D_\delta \setminus \{0\}) \cap \B_\e]  $
is a solid Milnor tube for $f \bar g$.
}

\medskip
\noindent
{\bf Corollary 2.}
 {\it
Let ${\mathcal L}_X:= X \cap \s_\e$ be the link of $X$, 
$V := {( f \bar g)}^{-1}(0)$ and  ${\mathcal L}_V:= {\mathcal L}_X \cap V$ be the link of $V$.
Then  one has a locally trivial fibration,
$$ \phi: {\mathcal L}_X \setminus {\mathcal L}_V \longrightarrow \s^1 \,,$$
which restricted to  ${\mathcal L}_X \cap  N(\e,\delta)$  is the natural projection $\phi = \frac{ f \bar g}{\vert  f \bar g \vert}$. }

\medskip
 In fact we know from \cite{PS2} that for $n=2$ the projection map $\phi$ in Corollary 2 
can be taken to be $\frac{f \bar g}{\vert f \bar g \vert}$ everywhere on ${\mathcal L}_X \setminus {\mathcal L}_V$, not only near the link of $V$. It would be interesting to know whether or not this statement holds also in higher dimensions.
  By \cite{CSS2}, this is equivalent to asking whether  all germs $f \bar g$ are $d$-regular (we refer to \cite{CSS2} for the definition); this is so when $n=2$, by \cite{PS2} and \cite{BP}.

We notice too that holomorphic map-germs actually have the stronger strict Thom $w_f$-property, by \cite{Par}  and
\cite[Theorem~4.3.2]{Brianson-Maisonobe-Merle:LSDSWCT}, even for functions defined on spaces with non-isolated singularities.
We do not know whether or not these statements extend to map-germs $f \bar g$   in general. Perhaps this can be proved using D. Massey's work \cite{Massey} about real analytic Milnor fibrations and a  {{\L}}ojasiewicz  inequality.

The authors are grateful to Arnaud Bodin for several useful comments and joyful conversations.

\section{The theorem}

Let $U$ be an open neighbourhood of the origin $\0$ in $\R^{m}$ and let
 $X \subset U$ be a real analytic variety  of dimension $n > 0$
with an isolated
singularity at  $0$. Let $\tilde f : (U,\0) \to (\R^k,0)$ be a real analytic map-germ which is
 generically a submersion, {\it i.e.}, its Jacobian
 matrix $D \tilde f$ has rank $k$ on a
dense  open subset of $U$. We denote by $f$ the restriction of $\tilde f $
to $X$.
As usual, we say that $x \in X \setminus \{0\}$ is {\it a regular point}
of $f$ if $Df_x:T_xX \to \R^k$
is surjective, otherwise $x$ is {\it a critical point}. A point $y \in \R^k$ is {\it a regular value} of $f$ if there
is no critical point in $f^{-1}(y)$; otherwise $y$ is
{\it a critical value}. We say that $f$ has
{\it an isolated critical value} at $0 \in \R^k$ if there is a neighbourhood $\D_{\delta}$
of $0$ in $\R^{k}$ so that all points $y \in \D_{\delta} \setminus \{0\}$ are regular values
of $f$.

Now let $U$ and $X$ be as before,   and let
$\tilde f : (U,\0) \to (\R^k,0)$ be  a real analytic map-germ such that $f = \tilde f|_X$  has
 an isolated critical value at $0 \in \R^k$. We set $V = f^{-1}(0) = \tilde f^{-1}(0) \cap X$.
According to \cite {Hi, LT}, there exist  Whitney stratifications
 of $U$ adapted to $X$ and $V$. Let $(V_\a)_{\a \in A}$ be such an stratification.

\begin{definition}
The  Whitney stratification $(V_\a)_{\a \in A}$ satisfies
{\it the Thom} $a_f$-{\it condition} with respect to  $f$
 if for every pair of strata $S_\alpha, S_\beta$ such that
$S_\alpha \subset {\overline S}_\beta$ and  $S_\alpha \subset V$, one has that for every
sequence of points
$\{x_k \} \in S_\beta$ converging to a point $x$
 such that the sequence
of tangent spaces $T_{x_k}(f^{-1}(f(x_k)) \cap {S_\beta}) $ has a limit $T$, one has that $T$ contains the
tangent space of $S_\a$ at $x$. We say that $f$ has {\it the Thom property} if such an stratification exists.
\end{definition}

Notice that this condition is automatically satisfied for strata contained in $V$, since in that case this regularity condition simply becomes Whitney's $(a)$-regularity.

Thom's property for complex analytic maps was proved by Hironaka in
\cite[Section 5 Corollary~1]{Hi}  for all holomorphic maps into $\C$ defined
on arbitrary complex analytic varieties. We remark that the  critical values of holomorphic maps are automatically isolated, while for real analytic maps into $\R^2$ this is a hypothesis we need to impose. We  refer to \cite{PS2} for examples of maps $f \bar g$ with  isolated  critical values, and also for examples with non-isolated critical values.
Hironaka's theorem was an essential ingredient for L\^e's fibration theorem
in \cite{Le1}. The corresponding statement was shown by L\^e D\~ung Tr\'ang to be  false in general for complex analytic mappings into $\C^2$ (see L\^e's example, for instance in \cite[p. 290]{Sab}). Similarly,  there are real analytic map-germs into $\R^2$ which do not have the Thom Property. Here we prove:

\begin{theorem}\label{Theorem}
Let $(X,\0)$ be a germ of an $n$-dimensional complex analytic set with an  isolated singularity and let $f , g : (X,\0) \rightarrow (\C,0)$ be holomorphic map-germs such that $f \bar g$ has an isolated critical value at $0 \in \C$. Then the real analytic germ $f \bar g$ has the  Thom $a_{f \bar g}$-property.
\end{theorem}

\begin{proof} The proof is inspired by the proof of Pham's theorem given in \cite{HL} (Theorem 1.2.1), which concerned holomorphic germs of functions defined on  complex analytic subsets of $\mathbbm  C^m$.

We first prove the theorem in  the case when the germ of $X$ at $\0$ is smooth,  {\it i.e.}, $X \cong \C^{n}$.

Let $U$ be an open neighbourhood of the origin $\0$ in $\C^{n}$ so that
 $f,g : U \rightarrow \C$  represent  the germs $f$ and $g$. We identify $\C^{n+ 1} \cong \C^n \times \C$ and denote by $(z_1, \ldots,z_{n+1})$ the coordinates in $\C^{n+1}$.

 Let us denote by $V$ the subset   in $\C^n$ with equation $fg =0$ and by  
 $Sing(f\bar g)$ the singular locus of $f \bar g$. Since $f \bar g$ has an isolated critical value at $0$, $Sing(f\bar g)$ is contained in $V$.

We need the following lemma:

\begin{lemma}   Let $G$ be the subset of $U \times \C$ with equation
$(f\bar g )(z_1,\ldots,z_n) - z_{n+1}^N = 0\,$,   $N \ge 1$.
Then the singular locus of $G$ is $Sing(f\bar g) \times \{0\}$.
\end{lemma}
\begin{proof} This follows by a straightforward computation of  the $2 \times 2(n+1)$ Jacobian matrix of $f \bar g - z_{n+1}^N$.
\end{proof}

Therefore, according to \cite{Wh} (just  as in \cite[1.2.4]{HL}  for the real analytic case), there exists a Whitney stratification $\sigma_N$ of $G$  such that  $G \cap (\C^n \times \{0\})= V  \times \{0\}$ is a union of strata and such that $G \setminus (V \times \{0\})$ is the union of the strata having dimension $2n$. We assume further that $\0$ is itself a stratum and that $U$ is chosen small enough so that every other stratum contains $\0$ in its closure.

Let $S_N$ be the stratification induced by $\sigma_N$ on $V \times \{0\}$. Adapting the arguments of \cite{HL}, we will prove that for $N$ sufficiently large, $S_N$ has the Thom condition with respect to $f \bar g$.

Assume that there exists a sequence of points $(x_k) = (z_1^{(k)}, \ldots, z_{n+1}^{(k)} )$ in $G \setminus (V \times \{0\})$ such that:

\begin{enumerate}
\item $\lim_{k \rightarrow \infty} x_k = x \in V \times \{0\}$,
\item If we set $t_k =  (f\bar g) (z_1^{(k)}, \ldots, z_{n}^{(k)} )$, then the sequence of $(2n-2)$-planes \\
$T_{x_k}\big((f\bar g)^{-1}(t_k) \times \{z_{n+1}^{(k)}\}\big)$ converges to a limit $T$ which does not contain the tangent space $T_xV_{\alpha}$ to the strata $V_{\alpha}$ of $S_N$ containing $x$.
\end{enumerate}

Moreover,  one can assume that  the sequence of $2n$-planes $T_{x_k}G$ converges to a limit $\tau$ since the Grassmanian of $2n$-planes in the Euclidian space is a compact manifold.

  For each $k$ one has
 $$T_{x_k}\big((f\bar g)^{-1}(t_k) \times \{z_{n+1}^{(k)}\}\big) \subset    T_{x_k}G,$$
 therefore $T \subset \tau$ and the intersection $\tau \cap (\C^n \times \{0\})$ has real dimension at least $2n-2$. Moreover, as $T_xV_{\alpha} \not\subset T$, one gets $T_xV_{\alpha} \neq \tau \cap \C^n \times \{0\}$.

But, since $\sigma_N$ satisfies Whitney's condition (a), one has $T_xV_{\alpha} \subset \tau$.   This implies that in fact  the intersection $\tau \cap (\C^n \times \{0\})$ has real dimension at least $2n -1$.
We will show that this is not possible if $N$ is sufficiently large.

According to  \cite{Lo}, there exists an open neighbourhood of $\0$ in $\C^n$ and a real number $\theta $,
$0<\theta  <1$, such that for each $z=(z_1,\ldots,z_n) \in \Omega$ one has :
$$\| (\hbox{grad} f)(z)\| \geq |f(z)|^{\theta} \ \hbox{ and } \| (\hbox{grad} g)(z)\| \geq |g(z)|^{\theta} $$
The Jacobian matrix of the map $f \bar g - z_{n+1}^N$ with respect to the
coordinates \\ $(z_1, \bar{z}_1, z_2, \bar{z}_2, \cdots, z_{n+1}, \overline{ z_{n+1}})$ in $\R^{2(n+1)}$ is the   $2\times 2(n+1)$ matrix  given in blocks by
$$D(f \bar{g})(z_1, \bar{z_1}, \ldots,z_{n+1}, \overline{z_{n+1}})=
 \left(
\begin{array}{ccccc}
M_1 & \ldots & M_i & \ldots & M_{n+1} \\
\end{array}
\right)\;,
$$
where for each $i = 1,\ldots,n$ the block $M_i$ is:
$$ M_i =  \left(
\begin{array}{cc}
  \frac{\partial (\Re(f \bar g))}{\partial z_i} &
\frac{\partial(\Re(f \bar g))}{\partial \bar z_i}
  \\
  &  \\
 \frac{\partial (\Im(f \bar g))}{\partial z_i} &
\frac{\partial(\Im (f \bar g))}{\partial \bar z_i}   \\
\end{array}
\right)\;,
$$
and
$$M_{n+1} = -\frac{N}{2} \left(
\begin{array}{cc}
z_{n+1}^{N-1} & \overline{z_{n+1}}^{N-1} \\ \, & \, \\
\frac{1}{i} z_{n+1}^{N-1}  & -\frac{1}{i}  \overline{z_{n+1}}^{N-1} \\
\end{array}
\right)\;.
$$
Then an easy computation leads to the following equation for the tangent space $T_{x_k}G$ at  $x_k= (z,z_{n+1}) \in G$ (we omit the $k$ in the coordinates in order to simplify the notations) :
$$
\sum_{i=1}^n \bigg( \frac{\partial f}{ \partial z_i}(z) \bar g (z)  v_i + \overline{\frac{\partial g}{ \partial z_i}}(z) \bar f (z) \overline{v_i} \bigg) -N z_{n+1}^{N-1} v_{n+1} = 0 \;.
$$

We consider the $2n$-vector appearing in the equation :
$$w_k(z) = \bigg( \frac{\partial f}{ \partial z_1}(z) \bar g (z),  \overline{ \frac{\partial g}{ \partial z_1}}(z), \ldots ,
\frac{\partial f}{ \partial z_n}(z) \bar g (z), \overline{ \frac{\partial g}{ \partial z_n}}(z) \bigg) \,.$$

For simplicity we omit to write that the functions below are evaluated at $(z)$.  We have:
$$\bigg(  \frac{\|w_k\|}{N|z_{n+1}|^{N-1}}\bigg)^2  =
   \frac{  |\bar g|^2 \sum_{i=1}^n \big|\frac{\partial f}{ \partial z_i} \big|^2 +|\bar f|^2 \sum_{i=1}^n \big|\frac{\partial g}{ \partial z_i} \big|^2  }  {N^2 |f \bar{g}|^{2 \frac{N-1}{N}}}  
    = 
\frac{   |\bar g|^2  \| \hbox{grad} f\|^2  +  |\bar f|^2  \| \hbox{grad} g\|^2}{N^2 |f \bar{g}|^{2 \frac{N-1}{N} }} \,.$$
Thus,
$$  \bigg(  \frac{\|w_k\|}{N|z_{n+1}|^{N-1}}\bigg)^2  = \frac{( |\bar g||\bar f|^{\theta})^2 +(|\bar f||\bar g|^{\theta})^2  }{ N^2 |f\bar g|^{2 \frac{N-1}{N}} }  \geq \frac{2}{N^2}  |f\bar g| ^{\theta - \frac{N-1}{N}} \,.$$

\medskip
When $N$ is sufficiently large, {\it i.e.}, $\theta - \frac{N-1}{N} <0$, one has :

$$  \lim_{k \rightarrow \infty}  \frac{\|w_k\|}{N|z_{n+1}|^{N-1}} = + \infty \,.$$
Therefore  the normalized limit of the vector $(w_k,  -N (z^{(k)}_{n+1})^{N-1})$ when $k \rightarrow \infty$, is a vector contained in $\C^n \times \{ 0\}$. Then the $2n$-plane $\tau$ contains the complex line $\{\0\} \times \C \subset \C^n \times \C$. This contradicts the fact that  $\tau \cap \C^n \times \{0\}$ has dimension at least $2n-1$. Thus, 
 if we set $t_k =  (f\bar g) (z_1^{(k)}, \ldots, z_{n}^{(k)} )$, then every sequence of $(2n-2)$-planes 
$T_{x_k}\big((f\bar g)^{-1}(t_k) \times \{z_{n+1}^{(k)}\}\big)$ that converges to a limit $T$ contains the tangent space $T_xV_{\alpha}$ to the strata $V_{\alpha}$ of $S_N$ containing $x$.  This completes the proof of the theorem when $X$ is smooth at $\0$.

 When $X \hookrightarrow \C^m$ has an
 isolated singularity at the origin, we take a Whitney stratification of a neighbourhood $U$ of $X$ in
$\C^m$ adapted to $X$ and to $V:= (f \bar g)^{-1}(0)$, and such that $\0$ is a stratum. We choose $U$ small enough such that any other stratum contains $\overline{0}$ in its closure. Now we consider a sequence of points $(x_k)$ in $X \setminus V$ converging to a point $x \in V$ and such that there is a limit $T$ of the corresponding sequence of spaces tangent to the fibers. If $x = \0$, then there is nothing to prove since $T$ contains the space tangent to this $0$-dimensional stratum. If $x \ne \0$, then we consider a coordinate chart $U_1$ for $X$ around $x$ and argue exactly as in the previous case, when $X$ was assumed to be smooth.
\end{proof}

We now look at 
  the corollaries. We know,  by Bertini-Sard's theorem in \cite{Ver}, that there is $\e >0$ such that all spheres in $\R^m$ centred at $\0$ with radius $\le \e$ meet  transversally each stratum in $\{f \bar g = 0 \}$.   Since  $f \bar g$ has Thom's $a_{f \bar g}$-property,  by Theorem \ref{Theorem}, we get that 
  given $\e >0$ as above, there exists $\delta >0$ sufficiently small with respect to $\e$ such that all fibers $(f \bar g)^{-1}(t)$ with $\vert t \vert \le \delta$ are transversal to the link $\mathcal L_X$. As usual, following the proof of  Ehresmann's fibration theorem (see for instance \cite{Mi2, Le1, PS2}),  this implies that 
  one has a locally trivial fibration $\,  N(\e,\delta) \buildrel {f} \over {\longrightarrow} \hbox {Image}(f \bar g) \subset \D_\delta \setminus \{0\} \,,
$
where  $N(\e,\delta) := [(f \bar g)^{-1}(\D_\delta \setminus \{0\})] \cap \B_\e  $
is a solid Milnor tube for $f \bar g$. 
Thus to complete the proof of Corollary 1 we must show that the image of $f \bar g$ covers all of $\D_\delta \setminus \{0\}$. This follows from the lemma below.

\begin{lemma} Let
$X, f$ and $g$ be as above, so that $f \bar g$ is not constant and it has an isolated critical value at $0 = f \bar g(\0)$. 
Then the germ $f \bar g$ is locally surjective at $\0$.
\end{lemma}

\begin{proof} If either $f$ or $g$ is constant, the statement in this lemma is a well-known property of holomorphic mappings. So we  assume  none of these maps is constant, neither is constant the map $f \bar g$. We may further assume that $f,g$ have no common factor, for otherwise we may divide both map-germs by that common factor and this will not change the image of the map $f \bar g$. We claim that 
since $f$ and $g$ are both holomorphic, we have that the  map-germ
$$(f, g) : \C^n \to \C \times \C $$
is locally surjective for all $n \ge 2$. That is, the image of every neighbourhood of $\0 \in \C^n$ covers a neighbourhood of $(0,0) \in \C \times \C$. In fact, for $n=2$ the map germ $(f,g)$ is a finite morphism, 
which is a finite covering map with ramification locus the discriminant
curve; so it is locally surjective. When $n \ge 3$ we may consider a generic complex 2-plane $\mathcal P$ in $\C^n$ which is transversal to the fibers of $(f,g)$
 and  apply the above arguments. Hence $(f, \bar g)$ is locally surjective, and so is $ f \bar g$.
\end{proof}

There are in \cite{CSS3}  examples of  analytic map-germs $(\R^n,\0) \buildrel{h} \over {\to} (\R^2,0)$ with an isolated critical value at $0$ which are not surjective. The image of $h$
misses a neighbourhood of an arc converging to $0$.
\medskip

The proof of Corollary 2 is just as that of Theorem 1.3 in \cite{PS2}, replacing the Milnor tube $[(f \bar g)^{-1}(\partial \D_\delta)] \cap \B_\e  $ by the solid Milnor tube $ [(f \bar g)^{-1}(\D_\delta \setminus \{0\})] \cap \B_\e  $, so we leave the details to the reader.
 (Compare with the first part of the proof of Theorem 1 in \cite{CSS2}).

\vskip0,3cm
\medskip

\medskip

Anne Pichon: Aix-Marseille Universit\'e, Institut de Math\'ematiques de Luminy UMR 6206 CNRS, Campus
de Luminy - Case 907, 13288 Marseille Cedex 9, France
\par\noindent
pichon@iml.univ-mrs.fr
\vskip0,2cm                                                                                                                                                                                                                     \medskip
Jos\'e Seade: Instituto de Matem\'aticas, Unidad Cuernavaca,
Universidad Nacional Aut\'onoma de M\'exico, A. P. 273-3,
Cuernavaca, Morelos, M\'exico.
\par\noindent
 jseade@matcuer.unam.mx


\begin{thebibliography}{999}

 \bibitem{AC} N. A'Campo.
{\em Monodromy of real isolated singularities.}
Topology 42(6),  (2003), 1229-1240.

\bibitem{BP}
A.~Bodin and A.~Pichon.
{\em Meromorphic functions, bifurcation sets and fibred links}.
 Math. Res. Lett., 14(3), (2007), 413--422.


\bibitem{BPS}
A.~Bodin, A.~Pichon, and J.~Seade.
{\em Milnor fibrations of meromorphic functions}.
  J. Lond. Math. Soc. (2), 80(2), (2009), 311--325.



\bibitem{Brianson-Maisonobe-Merle:LSDSWCT}
J. Brian{\c{c}}on, Ph. Maisonobe,  M. Merle.
{\em Localisation de syst\`emes diff\'erentiels, stratifications de
  Whitney et condition de Thom.}
 Invent. Math., 117(3), (1994), 531--550.


\bibitem{CSS2} J. L. Cisneros-Molina, J. Seade, J. Snoussi.
{\em Milnor Fibrations and $d$-regularity for real analytic
  Singularities}. Int. J.  Maths.  21 (4),  (2010), 419--434.


\bibitem{CSS3} J. L. Cisneros-Molina, J. Seade, J. Snoussi.
{\em Remarks on Milnor Fibrations  for real analytic
  Singularities}. Preprint 2011.


\bibitem{RT}    
R.~N.~A. dos Santos and M.~Tib{\u{a}}r.
\newblock Real map germs ahd higher open book structures.
\newblock Geom. Dedicata  147  (2010), 177--185.


\bibitem
 {HL} H. Hamm and D.T. L\^e.  {\em Un th\'eor\`eme de Zariski du type de Lefschetz}, Annales Scientifiques de  l'ENS, 4\`eme s\'erie, tome 6 (3) (1973), 317--355.

\bibitem
 {Hi} H. Hironaka.  {\em  Stratification and Flatness};   in
``Real and Complex singularities'', Proceedings of the Nordic Summer School,
Oslo 1976, Sijhoff en Nordhoff, Alphen a.d. Rijn
1977, 199--265.


\bibitem{Le1} D. T. L\^ e.  {\em Some remarks on the relative monodromy}, in ``Real and
Complex Singularities",  Proc.  Nordic Summer School,  Oslo 1976, Sijhoff en Nordhoff, Alphen a.d. Rijn
1977,  397--403.

\bibitem {LT} D. T. L\^e, B. Teissier.  {\em Cycles \'evanescents et conditions de {W}hitney},
Proc. Symp. Pure Math. 40 (Part 2), (1983), 65--103.

\bibitem{Lo}    S. Lojasiewicz, {\it Ensembles semi-analytiques}. I.H.E.S., Bures-Sur-Yvette, 1965.

\bibitem{Massey}
D.~B. Massey.
{\em Real analytic Milnor fibrations and a strong {{\L}}ojasiewicz
  inequality}, In ``Real and Complex Singularities", Proc.  Workshop on Real
  and Complex Singularities (2008), S\~ao Carlos, Brazil. Ed. M. Manoel et al. London Math. Soc. 
  Lecture Notes Series vol.  380 (2010).

\bibitem
{Mi2} J. Milnor. {\em   Singular points of complex hypersurfaces,} Ann. of Maths. Studies 61, 1968, Princeton Univ. Press.

\bibitem {Oka2}
M. Oka.
{\em Non-degenerate mixed functions}.
 Kodai Math. J. 33(1),  (2010), 1-62.

\bibitem{Par}
A. Parusi{\'n}ski.
{\em Limits of tangent spaces to fibres and the {$w\sb f$} condition.}
Duke Math. J., 72(1), (1993), 99--108.

\bibitem{Pichon} 
A. Pichon. {\em  Real analytic germs
$f \bar{g}$ and open-book decompositions of the $3$-sphere},  International
Journal of Mathematics, Vol. 16(1), (2005), 1-12.

\bibitem{PS2}
A.~Pichon and J.~Seade. {\em{Fibred multilinks and singularities $f \bar g$}.}
 Math. Ann., 342(3), (2008), 487--514.
 
 \bibitem{RS}
M.~A.~S. Ruas and R.~N.~A. dos Santos. 
{\em  Real {M}ilnor fibrations and (c)-regularity.}
 Manuscripta Math., 117(2), (2005), 207--218.


 
 \bibitem{RSV}
M.~A.~S. Ruas, J.~Seade, and A.~Verjovsky.
{\em On real singularities with a {M}ilnor fibration.}
 In {\em Trends in singularities}, Trends Math., pages 191--213.
  Birkh\"auser, Basel, 2002.

\bibitem{Sab} C. Sabbah. {\em Morphismes analytiques stratifiés sans éclatement et cycles évanescents}.  In  ''Analysis and topology on singular spaces, II, III (Luminy, 1981)``.    Ast\'erisque, 101-102, Soc. Math. France, Paris, 1983, 286 -- 319.


\bibitem{Se3}
J.  Seade.  \emph{On the topology of isolated singularities in analytic spaces}. Progress in Mathematics, 241. Birkh\"auser Verlag, Basel, 2006.

\bibitem{Ver}
J.-L. Verdier.
\newblock Stratifications de {W}hitney et th\'eor\`eme de {B}ertini-{S}ard.
\newblock {\em Invent. Math.}, 36, (1976), 295--312.

\bibitem {Wh} H. Whitney,   {\em  Tangent to analytic varieties}, Ann. of Math., Vol. 81 (1965), 496--549.



\end{thebibliography}
\end{document}